\documentclass[11pt,final]{article}
\usepackage{amsfonts}
\usepackage{amsmath}
\usepackage{amssymb}
\usepackage{amsthm}
\usepackage{cite}
\usepackage[margin=1.0in]{geometry}
\usepackage{amsthm}

\newtheorem{mypro}{Proposition}
\newtheorem{mytho}{Theorem}

\newtheorem{mydef}{Definition}

\def\bE{{\mathbb{E}}}
\def\bR{{\mathbb{R}}}

\def\btheta{\boldsymbol\theta}

\def\eps{\varepsilon}

\title{Parametric Estimation from Approximate Data: Non-Gaussian Diffusions}

\date{\today}

\author{
Robert Azencott
\thanks{Department of Mathematics,
University of Houston,
Houston, TX 77204-3008, (\texttt{razencot@math.uh.edu}).}
\and
Peng Ren
\thanks{Department of Mathematics,
University of Houston,
Houston, TX 77204-3008, (\texttt{pren@math.uh.edu}).}
\and
Ilya Timofeyev
\thanks{Department of Mathematics,
University of Houston,
Houston, TX 77204-3008, (\texttt{ilya@math.uh.edu}).}
}

\begin{document}
\maketitle
\begin{abstract}
We study the problem of parameters estimation in Indirect Observability contexts, where $X_t \in R^r$ is an unobservable stationary process parametrized by a vector of unknown parameters and all observable data are generated by an approximating process $Y^{\eps}_t$ which is close to $X_t$ in $L^4$ norm.
We construct consistent parameter estimators which are smooth functions of the sub-sampled empirical mean and empirical lagged covariance matrices computed from the observable data. We derive explicit optimal sub-sampling schemes specifying the best paired choices of sub-sampling time-step and number of observations. 
We show that these choices ensure that our parameter estimators reach optimized asymptotic $L^2$-convergence rates, which are constant multiples of the $L^4$ norm 
$|| Y^{\eps}_t - X_t ||$.
\end{abstract}


\section{Introduction}

The amount of available observational data has increased massively in recent years due to rapid technological 
advances in science and engineering. Often, it is desirable to fit an appropriate parametrized stochastic model to 
the available data and then use this model for forecasting, analysis, etc. Stochastic processes $X_t$ driven by 
systems of stochastic differential equations (SDEs) have often been used for this purpose so that both parametric and non-parametric techniques for fitting SDEs to the available data have a rich history (see \cite{kut2004,heyde97,rao99,Azen4} for a general overview).
Non-parametric approaches for SDE data modeling have used Bayesian methods as in \cite{bayes1,bayes2,bayes3,bayes4,ecs01}, exploited 
the spectral properties of the infinitesimal generator \cite{hst98,crove06b,crova11}, or have developed maximum likelihood function estimation as in \cite{kelly2004,sorensen97,kessler1999}, as well as drift and diffusion estimates by 
conditional expectations of process dynamics over short time intervals \cite{estim1,estim2,estim3,suba02,berner05,comte07}, with potential use of kernel based techniques as in 
\cite{baph03,zanten01}. 
For parametrized SDE models, various moments based parameter estimators (see \cite{gallant97} and references 
therein) have been implemented, as well as approximate maximum-likelihood parameters estimators after time 
discretization of the SDEs (see for instance \cite{as02,brsc02,pedersen95,gadhyan15}).
Minimum-contrast estimators have also been used for parametric estimation of diffusions \cite{dcd86,gcjl99}.
The asymptotic consistency and efficiency of SDEs model fitting to data have been well analyzed in the literature, 
although computational issues remain key questions in high dimensional applications. 

In this paper we address the 
problem of parameter estimation for multi-dimensional diffusions in a specific context, namely in situations when the 
process $X_t$ to be modeled is not directly observable (i.e. $X_t$ is \emph{unobservable}).
Instead, the observed data are generated by an approximating process $Y^\eps_t$ which involves a small 
\lq\lq{}scaling\rq\rq{} parameter $\eps$, and the SDEs driving $X_t$ are discovered by asymptotic analysis 
as $\eps \to 0$. Moreover, often the precise dynamics of the process $Y^\eps_t$ is not known or too complex 
to be explicitly formalized.
In such \emph{Indirect Observability} contexts, where the SDE system driving the unobservable $X_t$ is 
parametrized by an unknown vector $\btheta$, but the observable $Y^\eps_t$ is generated by unknown dynamics 
approximating the behavior of $X_t$, a natural goal is to efficiently estimate $\btheta$ from the observable data $Y^\eps_t$, assuming that for some adequate norm $Y^\eps_t \to X_t$ as $\eps \to 0$.
The main mathematical goal in this indirect observability context is to estimate the underlying SDE parameters 
by classes of estimators depending only on $\eps$ and on the observable data $Y^{\eps}_t$; of course 
as $\eps \to 0$, one wants these estimators to be consistent and approximately efficient. Recent references 
for consistent estimation of parametrized diffusions under indirect observability include 
\cite{past07,papast09,Azen1,Azen2,Azen3,crova11,dgcs14,disa14,gcjl99}.
In financial applications, indirect observability situations emerge as soon as dynamic models involve the (unobservable) volatilities  of assets, and their replacement by observable approximations of volatilities, such as 
\lq{}\rq{}realized volatilities\rq{}\rq{} computed from stock prices 
(see, e.g., \cite{Nielsen1, Nielsen2, Nielsen3, gcjl98, chpod12, boin02}). 
Indirect observability is also a natural context for stock prices dynamics  based on  precise noise microstructure 
(see, e.g. \cite{zms2005, smz2005}).

Our indirect observability framework covers a broad range of SDE systems driving unobservable multidimensional 
processes $X_t$, where the drift and diffusion terms are fairly generic smooth functions of an unknown parameter 
vector $\btheta$. Such SDEs are widely used in engineering and automatic control, population evolution, atmospheric 
and ocean dynamics, stock prices dynamics, options pricing, etc. to approximate the behavior of the leading variables 
of interest. To construct consistent estimators $\hat{\btheta}$ of $\btheta$ based on observed trajectories 
$Y^\eps_{[0,t]}$ of the approximate data, a natural approach is to first derive \lq\lq{}ideal\rq\rq{} estimators as 
specific functionals $\phi(X_{[0,t]})$ of the unobservable trajectory $X_{[0,t]}$, and then prove that under adequate 
conditions $\phi(Y^\eps_{[0,t]}) \to \btheta$ as $\eps\to 0$ and $N(\eps) \to \infty$, where $N(\eps)$ is the size of 
the observational data sample.
In \cite{Azen1, Azen2, Azen3}, we had combined this approach with data sub-sampling to generate consistent 
estimators $\hat{\btheta}$ based on approximate observable data $Y_t^\eps$ with fairly generic joint distributions, 
but for Gaussian limiting processes $X_t$, and we had also determined how nearly optimal sub-sampling rates should 
depend on $\eps$.
In this paper we extend our indirect observability analysis to stationary processes $X_t$ such that $Y^\eps_t \to X_t$ 
in $L^4$ as $\eps \to 0$, but with weak restrictions on the joint distributions of $X_t$ and $Y^\eps_t$, which can both 
be \emph{non-Gaussian}.

Discretized sub-sampling is standard to collect data from a continuous process. The observable data are then of the 
form $Y^\eps_{n \delta}$, where the observational time-step $\delta$ is determined by data acquisition protocols 
for sensors recordings or by the computational time-step for observables generated by numerical PDE models. As is 
well known, adequate data sub-sampling can reduce computational overhead without sacrificing estimators accuracy. 
So we introduce a user selected subsampling time-step $\Delta(\eps)$, which will be a multiple of $\delta$ and will 
fix the observational sample $\{Y^\eps_{n \Delta(\eps)},~ n=1,\ldots,N(\eps)\}$ retained to estimate ${\btheta}$. 
We then adress the  key issue of how to optimize parameter estimators performance by 
seeking the \lq\lq{}best\rq\rq{} asymptotic  
choices for the number of observations $N(\eps)$ and the sub-sampling time-step $\Delta(\eps)$. 
We thus derive explicit 
relations between $\Delta(\eps)$, $N(\eps)$, and $\rho(\eps) = \| Y^\eps_t - X_t\|_4$ ensuring nearly optimal behavior for parameter estimation errors as $\eps \to 0$. This is achieved by focusing on estimators which are (not necessarily explicit) smooth functions of empirical lagged moments up to order two computed from the $N(\eps)$ 
sub-sampled observables $\{Y^\eps_{n \Delta(\eps)},~ n=1,\ldots,N(\eps)\}$. Since in practice the available numbers of observations remain moderate, our nearly optimal  choices for the pair $(N(\eps), \Delta(\eps))$ are constructed to simultaneously minimize the  size of parameter estimation errors and the computational/observational complexity.

The paper is organized as follows. Basic assumptions about our indirect observability setup are given in section \ref{IOS}. Our  main results about speed of convergence for parameters estimators based on observable data and the associated characterization of optimal sub-sampling schemes are presented in section \ref{sec:indirect}. 
In section \ref{parestim} we outline the main class of parameter estimators  studied here, namely smooth functions of lagged moments of order $\leq 2$.
Section \ref{consistentUnobservable} contains key technical results on $L^2$ consistency of
moments estimators computed from  unobservable data.
Potential applications of our results to stationary multi-dimensional diffusions and to Heston SDEs are discussed in the two sections  \ref{diffusions} and  \ref{sec:appl}.

\section{Indirect Observability Setup}
\label{IOS}
\subsection{Basic Indirect Observability Hypotheses} \label{BIOH}
\paragraph{Notations.} For any matrix $M = ( M_{i, j} ) $, we set 
$\| M \|= \sup_{i, j} M_{i, j}$, and $M^*$ is the transpose of $M$.
The $L^p$ norm of a random matrix $M$ is denoted by $\| M \|_p = \bE(\|M\|^p) ^{1/p}$.

\paragraph{Indirect Observability.} 
Our formal \emph{indirect observability setup} $X_t = \lim_{\eps \to 0} Y^\eps_t $ involves \\
-- a set of \emph{directly observable} continuous time stochastic processes $Y^\eps_t \in \bR^r$ indexed by 
a small ``scale'' parameter $\eps >0$, with $L^4$ norms $\| Y_t^\eps\|_4$ uniformly bounded for all $\eps >0$ 
and $t \geq 0$, \\
-- an \emph{unobservable} strictly stationary process $X_t = \lim_{\eps \to 0} Y^\eps_t $, where convergence holds at 
uniform $L^4$-speed $\rho(\eps)\to 0 \text{~as~} \eps\to 0$, so that
\begin{equation} \label{L4speed}
\| Y_t^\eps - X_t \|_4 \leq \rho(\eps), \text{~for all~} t \geq 0, \, \eps > 0.
\end{equation}
Moreover, we will assume that \\
-- the mean $\mu$ of $X_t$ and the lagged covariance matrices $K(u) = \bE (X_t X_{t+u} ^*)- \mu \mu^*$ are $C^1$ functions of an \emph{unknown} 
parameter vector $\btheta \in \Theta$, with $\Theta$ open in $\mathbb{R}^p$, \\
-- as functions of the time lag $u $, the matrices $K(u)$ are locally Lipschitz, uniformly in $\btheta$.

Since $X_t$ is \emph{not} directly observable, our main goal here is to construct \emph{observable estimators} 
$\btheta^{\eps} $ of $\btheta$, i.e. estimators depending only on the observable data $Y_t^{\eps}$, and hence 
of the form 
\begin{equation}\label{def12}
\btheta ^{\eps} = \phi^{\eps} ( Y^\eps_{t_1}, \ldots , Y^\eps_{t_q} ),
\end{equation}
where the number $q$ of observables and the instants $t_1, \ldots, t_q$ depend on $\eps$, and each Borel 
function $\phi^\eps : \bR^q \to \bR^m$ is deterministic. As $\eps \to 0$, achieving consistency in probability for the estimators $\btheta^{\eps}$ will require to specify nearly optimal choices for $q(\eps)$ and for the time 
grid $t_1(\eps), \ldots, t_q(\eps)$, and to clarify how these choices depend on the approximation speed 
$\rho(\eps)$.

\subsection{Multi-dimensional Diffusions under Indirect Observability} 
\label{diffusions}
\paragraph{Stationary multi-dimensional diffusions.} 
Many practical examples of indirect observability are linked to the 
approximation of high dimensional multiscale systems by reduced stochastic differential equations obtained by averaging and homogenization \cite{book:past08, mtv6, melst11, jative14, past07, papast09, crova11, Azen3}. For instance this is  the case of  homogenization applications to  the atmosphere-ocean science \cite{mtv2, mtv3, frmava05, frma06, dat13}. In those contexts 
the fully explicit mathematical dynamics for the observable $Y^\eps_t$ is typically too complex to be fully modeled. However, the unobservable limit process $X_t \in \bR^r $ is often modeled by a relatively low-dimensional \emph{strictly stationary multi-dimensional diffusion} driven by an SDE system with coefficients depending smoothly 
on the parameter vector $\btheta \in \Theta \subset \bR^p$. These SDE systems, derived by an ad-hoc analysis 
of the main \lq\lq{}mechanisms\rq\rq{} generating the $Y^\eps_t$ observables, are of the form
\begin{equation} \label{SDE}
dX_t = b(X_t, \btheta) dt + \sigma(X_t, \btheta) dW_t,
\end{equation}
where \\
-- $W_t$ is a multi-dimensional Brownian motion, \\
-- the drift $b(x, \btheta)$ and the matrix $\sigma(x, \btheta)$ are $C^{\infty}$ functions of $x \in \bR^r $ 
and $\btheta \in \bR^p$, \\
-- the matrix $a(x, \btheta) = \frac{1}{2} \sigma \sigma*$ is invertible for all $x $ and $\btheta$. \\
The transition density $p_{\btheta} (t,x,y)$ of $X_t$ then depends smoothly on $\btheta$ 
\cite{stva05}
and is the fundamental solution of a parabolic PDE with coefficients $\frac{1}{2} a$ and $b$. 
The literature (see for instance \cite{lsu,iko,fr}) has extensively discussed these Fokker-Plank PDEs, as well as 
the elliptic PDE verified by the stationary density $q_{\btheta} (x,y) = \lim_{t \to \infty} p_{\btheta} (t,x,y)$, 
when $X_t$ is strictly stationary. 
In dimension 1, strict stationarity of $X_t$ is equivalent to integrability in $x \in \bR$ of $p(x) = \frac{1}{a(x)}exp(\frac{b(x)}{a(x)})$ and (see \cite{book:karlintaylor}) 
the stationary density of $X_t$ is then proportional to $p(x)$.   
In higher dimensions, the literature does not seem to provide easy to use sufficient conditions for stationarity, so that 
stationarity has to be verified in each specific application.

\paragraph{Parameter estimation for diffusions.} The abundant literature on parameter estimation for 
multi-dimensional diffusions such as $X_t$ 
(see overviews \cite{kut2004,heyde97,rao99} and references therein) 
offers a broad range of parameter estimators which are non explicit but numerically computable functions
\begin{equation} \label{gq}
\hat{\btheta} = g_q ( X_{t_1}, \ldots, X_{t_q} )
\end{equation}
of $q$ diffusion \lq\lq{}data\rq\rq{} $X_{t_1}, \ldots, X_{t_q}$. Estimators of this type can for instance be 
numerically derived by approximate maximum likelihood after time discretization 
(see, e.g., \cite{gadhyan15,pedersen95,as02,ecs01,rao99,heyde97,kut2004,cat10}).

Under variously formulated sufficient conditions on the diffusion, maximum likelihood estimators of $ \btheta$ 
become asymptotically consistent for $q$ sufficiently large and for dense enough specific time grids 
${t_1, \ldots, t_q}$. However in our indirect observabilty setup, these estimators are obviously not ``observable'', 
since they involve the non-observable diffusion data $ X_{t_1}, \ldots, X_{t_q} $. When \lq\lq{}good\rq\rq{} 
choices are known for the $g_q$ functions in \eqref{gq}, one can naturally attempt to construct consistent 
observable estimators $\btheta ^{\eps}$ by setting 
\[
\btheta ^{\eps} = g_q ( Y^\eps_{t_1}, \ldots , Y^\eps_{t_q} ),
\]
with \lq{}\lq{}adequate\rq{}\rq{} 
choices for the number of observations $q(\eps)$ and for the time grid $t_1(\eps) , \ldots, t_q(\eps)$.
A key technical goal in our paper is to 
\lq{}\lq{}optimally\rq{}\rq{} select  $q(\eps)$ and this time grid  as   functions of $\eps$, and to link such 
nearly optimal sub-sampling schemes to the $L^2$-speed of approximation $\rho(\eps)$ of $X_t$ by the 
observable $Y^\eps_t$.

\paragraph{Sparsely parametrized stationary diffusions.} For ``uniformly elliptic'' stationary diffusions $X_t$ driven 
by equation \eqref{SDE}, the matrix $a$, its inverse $a^{-1}$, and the drift $b$ are classically assumed to be 
uniformly bounded for $x \in \bR^r $ and $\btheta \in \bR^p$. 
The well known Aronson bounds of the diffusion transition density $p_{\btheta}(t,x,y)$ (see \cite{aron}) then imply the finiteness of all lagged moments of arbitrary order for $X_t$, and due to the smoothness of $p_{\btheta}$, 
all these lagged moments are then necessarily smooth functions of $\btheta$.

\paragraph{Conjecture.} On the basis of multiple concrete examples of stationary diffusions, we conjecture that 
when the matrices of diffusion coefficients $a$ and $b$ in equation \eqref{SDE} are \emph{analytic functions} of 
the $p$-dimensional parameter vector $\btheta$, one can then find a vector $\Psi = \left[ \Psi_1, \ldots, \Psi_p \right]$ 
of $p$ lagged moments of order $\leq 2$ of $X_t$, which uniquely determine $\btheta$ as a (non explicit) smooth 
function $\btheta = G(\Psi)$ of $\Psi$.
For such ``sparsely parametrized'' diffusions, parameter estimators based on estimation of lagged moments become natural targets, as outlined in the next section.

\section{Parameter estimators based on lagged moments of order $\leq 2$}
\label{parestim}

\subsection{Sparsely parametrized stationary processes}
\label{sparse.parametrization}
The preceding conjecture on multidimensional stationary diffusions leads us to introduce the following definition.
\begin{mydef}
\label{sparse.def} 
We say that a strictly stationary process $X_t$ is \emph{sparsely parametrized} by $\btheta \in \bR^p$ whenever 
one can find a vector $\Psi = \left[ \Psi_1, \ldots, \Psi_p \right]$ of $p$ lagged moments of order $\leq 2$ of $X_t$, 
which uniquely determine $\btheta$ as a smooth function $\btheta = G(\Psi)$ of $\Psi$.
\end{mydef}
Let $X_t$ be a sparsely parametrized stationary process embedded in an indirect observability setup 
$X_t = \lim_{\eps \to 0}Y^\eps_t $. As $\eps \to 0$, constructing consistent \textit{observable estimators} for 
all lagged moments of order $\leq 2$ of $X_t$ is then clearly equivalent to constructing consistent observable 
estimators $\hat{\btheta}^\eps$ for $\btheta$. Indeed if $\btheta = G(\Psi)$, one can simply set 
$\hat{\btheta}^\eps = G(\hat{\Psi}^\eps)$, where $\hat{\Psi}^\eps$ is an observable estimator of the vector of 
moments $\Psi$. Observable estimators for the lagged 2nd order moments $\Psi_j$ are naturally provided by the 
empirical lagged covariances of $Y^\eps_t$, and hence consistent observable estimators of $\btheta$ can then 
be constructed as smooth functions of empirical lagged moments of order $\leq 2$ of the observable data $Y^\eps_t$.

\subsection{Sub-sampled Empirical Moments}
As pointed out in many papers on indirect observability (see, e.g., \cite{past07,papast09,Azen3,gcjl98,crova11}), 
correctly scaled sub-sampling of observable data can reduce computational overhead and decrease the 
bias of parameter estimators. 
A main technical point here will then be to select nearly optimal functions of $\eps$ for the sub-sampling time 
step $\Delta(\eps)$ of observable data as well as for the overall number $N(\eps)$ of sub-sampled observables.
So we now define adaptive sub-sampling schemes for empirical estimators of lagged covariances. 
\begin{mydef}
\label{adaptive sampling}
\emph{Adaptive sub-sampling schemes} will be specified by two functions of $\eps$: the sub-sampling time step 
$\Delta = \Delta(\eps) > 0$ and the number $N = N(\eps) $ of sub-sampled observations, and we will impose 
the natural conditions
\begin{equation}\label{N Delta}
\Delta(\eps) \to 0 \quad \text{and} \quad N(\eps) \Delta(\eps) \to \infty \; \quad \text{as} \quad \eps \to 0.
\end{equation}
\end{mydef}
Fix an adaptive sub-sampling scheme $\Delta(\eps), N(\eps)$. Each time lag $u$ will then be approximated by 
$\kappa \Delta$ with the integer $ \kappa$ given by
\begin{equation}\label{kappa}
\kappa= \kappa (u, \eps) = \bigg[ \frac{u}{\Delta(\eps)} \bigg] = \;\; \text{closest integer to} \; \; 
\frac{u}{\Delta(\eps)}.
\end{equation}
Note that $\kappa(0,\eps) = 0$, and that \eqref{N Delta} implies 
\[
\lim_{\eps \to 0} \kappa(u,\eps) = \infty \quad \text{and} \quad 
\lim_{\eps \to 0} \frac{\kappa(u,\eps)}{N(\eps)} \to 0.
\]
We estimate the mean $\mu$ of $X_t$ by the (observable) empirical mean 
\begin{equation}\label{barY}
\bar{Y}^\eps = \frac{1}{N}\sum_{n=1}^{N} Y^\eps_{n \Delta}.
\end{equation}
\begin{mydef}
\label{def4}
For any given time lag $u$, let $\kappa = \kappa(u,\eps) $ be as in equation \eqref{kappa}. Define the time-shifted empirical mean 
$ \tau_u \bar{Y}^\eps $ by 
\begin{equation}\label{taubarY}
\tau_u \bar{Y}^\eps = \sum_{n = 1}^{N } \; Y^\eps_{( n + \kappa ) \Delta }.
\end{equation}
We then estimate the lagged covariance matrix $K(u)$ of $X_t$ by the observable sub-sampled empirical covariances 
\begin{equation}\label{DefCovYep}
\hat{K}^\eps_Y(u) = \frac{1}{N} \, 
\left[ \; \sum_{n = 1}^{N } \; Y^\eps_{n \Delta} \, ( Y^\eps_{( n + \kappa ) \Delta})^* \; \right] - 
\bar{Y}^\eps ( \tau_u \bar{Y}^\eps)^*,
\end{equation}
where $N= N(\eps)$, $\Delta= \Delta(\eps)$, $\kappa= \kappa (u, \eps)$
\end{mydef}
We always assume that $\eps$ is small enough to force $N \gg \kappa$.

\subsection{Sensitivity of Observable Covariances Estimators to Data Approximation}
In formulas \eqref{barY} and \eqref{DefCovYep}, replacing all observable $Y^\eps_t$ by their limits $X_t$ transforms the observable estimators $\bar{Y}^\eps$, $\tau_u \bar{Y}^\eps$ and $\hat{K}^\eps_Y(u)$ into \textit{unobservable} estimators $\bar{X}^\eps$, $\tau_u \bar{X}^\eps$ and $\hat K^\eps_X(u)$ given by
\begin{eqnarray}
\bar{X}^\eps = \frac{1}{N}\sum_{n=1}^{N} X_{n \Delta}, 
\label{barX} \\
\tau_u \bar{X}^\eps = \frac{1}{N} \sum_{n = 1}^{N } \; X_{(n +\kappa) \Delta },
\label{taubarX}
\end{eqnarray}
and
\begin{equation}\label{DefCovXep}
\hat{K}^\eps_X(u) = 
\frac{1}{N} \left[ \; \sum_{n = 1}^{N } \; X_{n \Delta} \, X_{( n + \kappa ) \Delta }^* \; \right] 
- \bar{X}^\eps ( \tau_u \bar{X}^\eps)^* = 
\frac{1}{N} \sum_{n = 1}^{N } \; (X_{n \Delta} - \bar{X}^\eps ) \, (X_{(n +\kappa) \Delta} - \tau_u \bar{X}^\eps )^* ,
\end{equation}
where, as before, $N= N(\eps), \Delta= \Delta(\eps), \kappa= \kappa (u, \eps)$. Obviously, the covariance estimator 
remain unchanged when $X_t$ is replaced by the centered process $X_t - \mu$.
We now evaluate the $L^2$-norm perturbations induced on $\bar{Y}^\eps $ and $\hat{K}^\eps_Y(u)$ when the 
observable data are replaced by their unobservable limits.
\begin{mypro}\label{ThCompareYX}
Consider any indirect observability setup $X_t = \lim_{\eps \to 0}Y^\eps_t $ as in section \ref{IOS}. Let $\rho(\eps)$ 
be the $L^4$-speed of convergence of $Y^\eps_t $ to $X_t$. Let $\nu$ be an upper bound for all the $L^4$ norms 
$\| X_t \|_4$ and $\| Y^\eps_t \|_4$. 
Fix an adaptive sub-sampling scheme $\Delta(\eps), N(\eps)$ verifying \eqref{N Delta}. Then as $\eps \to 0$, the 
observable estimators $\bar{Y}^\eps $ and $\hat{K}^\eps_Y(u)$ defined by \eqref{barY} and \eqref{DefCovYep} 
and their unobservable versions $\bar{X}^\eps$ and $\hat{K}^\eps_X(u)$ verify 
\begin{eqnarray}
\| \hat K_Y^\eps(u) - \hat K_X^\eps(u) \|_2 \leq 4 \nu \rho(\eps) \quad \text{for all}\quad u \geq 0, \label{CompKXKY} \\
\| \bar{Y}^\eps - \bar{X}^\eps \|_4 \leq \rho(\eps) \quad \text{and} \quad
\| \tau_u \bar{Y}^\eps - \tau_u \bar{X}^\eps \|_4 \leq \rho(\eps). \label{CompbarXbarY}
\end{eqnarray}
\end{mypro}
Proposition \ref{ThCompareYX} reduces the $L^2$-consistency analysis for our \emph{observable} estimators of lagged moments to $L^2$-consistency 
analysis for the \textit{unobservable} moments estimators based on the $X_t$ data and given by \eqref{barX}, \eqref{taubarX}, \eqref{DefCovXep} 
\begin{proof}
We extend a similar proof given in \cite{Azen3}. The uniform bounds on $\| Y_t^\eps -X_t \|_4$, $\| X_t \|_4$, $\| Y^\eps_t \|_4$ are 
preserved by convex linear combinations, which yields 
\begin{equation}\label{barYminusX}
\| \bar{Y}^\eps \|_4 \leq \nu, \; \; 
\| \bar{X}^\eps \|_4 \leq \nu, \; \; 
\| \bar{Y}^\eps -\bar{X}^\eps \|_4 \leq \rho(\eps) , \; \; 
\| \tau_u \bar{Y}^\eps - \tau_u \bar{X}^\eps \|_4 \leq \rho(\eps).
\end{equation}
Note that for any two random column vectors $A,B \in \bR^m$, one has the elementary inequalities
\begin{equation} \label{norm.AB}
\| A B^* \|_2 \leq \| A \|_4 \| B \|_4 \quad \text{and} \quad \| A^* B \|_2 \leq m \| A \|_4 \| B \|_4 
\end{equation}
Let $V, W, V', W'$ be four random column vectors in $\bR^m$, with $L^4$-norms inferior to $\nu$, and verifying
\begin{equation}\label{lemmahyp}
\| V - V' \|_4 \leq \rho \quad \text{and} \quad \| W - W' \|_4 \leq \rho.
\end{equation}
By \eqref{norm.AB}, the norms $\| V ( W - W' )^* \|_2 $ and $ \| (V-V') (W')^* \|_2 $ are resp. inferior to $\| V \|_4 \| W - W' \|_4 $ and $\| W' \|_4 \|V - V' \|_4 $, so that 
\[
\| V ( W - W' )^* \|_2 \leq \nu \rho \quad \text{and} \quad \| (V-V') (W') ^* \|_2 \leq \nu \rho.
\]
Hence, the random matrices $V W^* $ and $V' (W')^* $ verify 
\begin{equation}\label{lemmaVW}
\| V W^* - V' (W')^* \|_2 = \| V ( W - W' )^* + (V-V') (W') ^* \|_2 \leq 2 \nu \rho .
\end{equation}
The bound \eqref{lemmaVW} can be applied when $V,W,V', W'$ are replaced by 
$ \bar{Y}^\eps, \tau_u \bar{Y}^\eps, \bar{X}^\eps, \tau_u \bar{X}^\eps$, and this yields
\begin{equation}\label{barYY}
\| \bar{Y}^\eps (\tau_u \bar{Y}^\eps) ^* - \bar{X}^\eps (\tau_u \bar{X}^\eps) ^* \|_2 \leq 2 \nu \rho(\eps) .
\end{equation}
The bound \eqref{lemmaVW} and relation \eqref{CompbarXbarY} similarly show that
\begin{equation}
\| Y^\eps_{n \Delta} (Y^\eps_{( n + \kappa ) \Delta }) ^* 
- X_{n \Delta} X_{( n + \kappa ) \Delta } ^* \|_2 \leq 2 \nu \rho(\eps) 
\end{equation}
for all $n, \Delta, \kappa$. Convex combinations preserve this uniform $L_2$-bound, and hence 
\begin{equation}\label{sumYY}
\left| \left| \frac{1}{N} \sum_{n =1}^N \left[ Y^\eps_{n \Delta} (Y^\eps_{( n + \kappa ) \Delta }) ^* 
- X_{n \Delta} X_{( n + \kappa ) \Delta } ^* \right] \right| \right|_2 \leq 2 \nu \rho(\eps) .
\end{equation} 
By definitions of $\hat K_Y^\eps(u)$ and $\hat K_X^\eps(u)$, the inequalities \eqref{barYY} and \eqref{sumYY} conclude the proof.
\end{proof}

\section{$L^2$-consistency of Unobservable Sub-sampled Moments Estimators}
\label{consistentUnobservable}
Since proposition \ref{ThCompareYX} effectively controls in $L^2$-norm the difference between 
observable and unobservable estimators of lagged covariances, we now focus on the consistency of 
sub-sampled empirical covariance estimators based on the \textit{unobservable} 
$X_t$. This will require assuming a fast enough asymptotic decorrelation speed for the process $X_t$.

\subsection{Integrable Decorrelation Rates for 4th Order Moments} 
\label{decorrelation.rates}
\begin{mydef} 
Let $X_t \in \bR^r $ be a stationary process with uniformly bounded moments of order 4.
Denote by $X_t(i)$, $i=1 \ldots r$ the coordinates of $X_t$. For any time interval $U \subset R^+$, let $F_U$ 
be the set of all random variables of the form $X_s(i)$ or $X_s(i) X_t(j) $ for arbitrary $ s, t \in U$ and 
$1 \leq i,j \leq r$.
Let $f(T) > 0$ be a decreasing continuous function of $T>0$ with finite integral $I(f) = \int_1^{\infty} f(T) dT$.
We say that $X_t$ has \emph{integrable decorrelation rate} $f(T)$ if for any disjoint time intervals $U$ and $V$ 
and for any random variables $G \in F_U$ and $H \in F_V$ one has 
\begin{equation}\label{decay4}
\bigg| \mathbb{E}(G H) - \mathbb{E}(G) \mathbb{E}(H) \bigg| \leq f(T) 
\end{equation}
where $ T = \min_{v \in V, u \in U}(v-u)$ is the \emph{time gap} between $U$ and $V$.
\end{mydef}

Definition \eqref{decay4} controls up to moments of order 4 the \lq\lq{}decorrelation\rq\rq{} rate between 
$X_t$ and $X_{t +u}$ when $u \to \infty$, and in particular implies that the lagged covariances matrices 
$K (u)$ of the process $X_t$ decay to zero at rate $f(u)$ when $u \to \infty$. 
The converse is of course not true for non-Gaussian processes $X_t$, but when $X_t$ is a Gaussian with lagged 
covariances $K(u)$ decaying to zero at rate $f(u)$, then $X_t$ necessarily has integrable decorrelation 
rate proportional to $f$, as can be seen from the classical formula
\[
E(Z_1 Z_2 Z_3 Z_4) = \sigma_{1, 2}\sigma_{3, 4} + \sigma_{1, 3}\sigma_{2, 4} 
+ \sigma_{2, 3}\sigma_{1, 4} ,
\] 
where the four random variables $Z_m$ are centered and jointly Gaussian with covariances $\sigma_{m, n}$.

The decorrelation condition \eqref{decay4} is also linked with classical mixing contexts.
Let $\mathcal{F}(U)$ be the set of random variables measurable with respect to the sigma-algebra generated 
by all the $ X_s $ with $s \in U$. 
To extend \eqref{decay4} to random variables $G \in \mathcal{F}(U)$ and $H \in \mathcal{F}(V)$ it is 
necessary to require a stronger \lq\lq{}mixing\rq\rq{} property for $X_t$. Recall that when $X_t$ is an ergodic 
stationary process having the $\phi$-mixing property (see, for instance, \cite{brad05} for a recent survey), 
there is a fixed decay rate $\phi(T) >0$ tending to 0 as $T \to \infty $, such that for any disjoint time 
intervals $U$ and $V$ separated by a time gap $T> 0$, and for any pair of events $A, B$ verifying 
$1_A \in \mathcal{F}(U)$ and $ 1_B \in \mathcal{F}(V)$ one has 
\[
| P(B \; | \; A) - P(B) | \leq \phi(T) .
\]
Provided the $X_t$ are in $L^4$, these uniform dependency decay rates will typically imply the validity of 
our condition \eqref{decay4} for some decorrelation function $f(T)$ deduced from $\phi(T)$.

\subsection{Accuracy of Unobservable Sub-sampled Empirical Covariance Estimators}
\label{sec:accuracy}
We now prove $L^2$- consistency for the \textit{unobservable} sub-sampled empirical estimators of 
lagged covariance matrices. As could be expected from earlier results, 
integrable decorrelation rate for $X_t$ plays here s a crucial role.
\begin{mypro}\label{ThUnobservable}
Let $X_t \in \bR^r$ be a stationary process with finite $L^4$- norm $\nu = \|X_t \|^4 $, and with
lagged covariance matrices $K(u)$ locally Lipschitz in $u$. Assume that $X_t$ has integrable decorrelation 
rate $f(T)$ as in \eqref{decay4} and set $I(f) = \int_0^{+ \infty} f(T) dT $. 
Fix any adaptive sub-sampling scheme $\Delta(\eps), N(\eps)$ verifying \eqref{N Delta}. 
Let $\hat{K}^\eps_X(u)$ be the (unobservable) 
sub-sampled empirical estimators of $K(u)$ defined by \eqref{DefCovXep}. 
Then the following two statements hold. \\
(I) For time lags $u$ in any bounded interval $[0,A]$, 
the estimators $\hat{K}^\eps_X(u)$ are consistent in $L^2$ as $\eps \to 0$, with uniform speed of 
convergence given by
\begin{equation}\label{covboundX}
\| \hat K_X^\eps(u) - K(u) \|_2 \leq \gamma \left[ \frac{1}{\sqrt{N \Delta}} + \Delta \right],
\end{equation}
where the constant $\gamma$ depends only on $A, I(f),\nu$. \\
(II) As $\eps \to 0$, the (unobservable) estimators $\bar{X}^\eps $ given by \eqref{barX} converge in $L^4$ and in $L^2$ to the mean 
$\mu$ of $X_t$ with the speed of convergence given by 
\begin{equation}\label{L2L4speed}
\| \bar{X}^\eps - \mu \|_4 \leq \frac{c}{(N \Delta)^{1/4}} \;\; \text{and} \;\; \| \bar{X}^\eps - \mu \|_2 \leq \frac{c}{(N \Delta)^{1/2}},
\end{equation}
where the constant $c$ depends only on $I(f)$ and $\nu$.
\end{mypro}
\begin{proof} 
The proof relies on the meticulous use of natural techniques, and is hence given in the appendix \ref{ProofUnobservable}. 
\end{proof}
In proposition \ref{ThUnobservable}, our key technical target was to identify how $L^2$ speeds of convergence 
depend on the adaptive scheme $N(\eps)$, $\Delta(\eps)$, in order to optimize our adaptive subsampling schemes, 
first for our unobservable empirical estimators of lagged covariances (see next corollary), and further on for the 
corresponding observable covariances estimators.

\paragraph{Similarity Notation.} Next, we introduce similarity notation for the limiting behavior of any two functions 
depending on a small parameter $\eps$. In particular, for any two functions $G(\eps)$ and $H(\eps)$, we write $G \sim H$ whenever $ \lim_{\eps \to 0} \frac{G(\eps)} {H(\eps)}$ is finite and strictly positive.

The next proposition determines the optimal relationship between the number of observations $N(\eps)$ and the 
sub-sampling time-step $\Delta(\eps)$. Smaller values of $\Delta(\eps)$ will not improve the limiting convergence rate,
but will lead to oversampling and, thus, unnecessary data acquisition and storage.
\begin{mypro}\label{PropCorollary} Let $X_t \in \bR^r$ be a stationary process in $L^4$, with locally Lipschitz lagged 
covariances $K(u)$ and integrable decorrelation rate $f(T)$. Select arbitrary numbers of observations $N(\eps)$ 
tending to $\infty$ as $\eps \to 0$ and define the sub-sampling time steps by
\begin{equation} \label{optimalDelta}
\Delta(\eps) \sim \frac{1}{N(\eps)^{1/3}}.
\end{equation}
The (unobservable) estimators $\hat K_X^\eps(u) $ associated to $\Delta(\eps), N(\eps)$ converge to the 
true $K(u)$ as $\eps \to 0$, with the following $L^2$ speed, valid for all $0 \leq u \leq A$,
\begin{equation} \label{optimalspeed}
\| \hat K_X^\eps(u) - K(u) \|_2 \leq C \frac{1}{N(\eps )^{1/3}},
\end{equation}
where $C$ is a constant determined by $A, I(f), \nu$.
Given the function $N(\eps)$ , the $L^2$ bounds in \eqref{optimalspeed} define, up to multiplicative constants, the 
\textit{best $L^2$ speed of convergence } obtainable under the generic assumptions of proposition \ref{ThUnobservable}.
\end{mypro}
\begin{proof} 
For $\| \hat K_X^\eps(u) - K(u) \|_2 $, proposition \ref{ThUnobservable} provides a bound proportional to 
$$
B(\eps) = \frac{1}{\sqrt{N \Delta}} + \Delta.
$$
For any given $N(\eps)$, the choice $\Delta(\eps) \sim N^{-1/3} (\eps)$ 
obviously minimizes $B(\eps)$, and one then has $B(\eps) \sim N^{-1/3}(\eps)$. 
This proves equation \eqref{optimalspeed}.

To check that the bound in equation \eqref{optimalspeed} cannot be improved further, one simply needs 
to construct a process $X_t$ verifying all the hypotheses of proposition \ref{ThUnobservable}, and such that 
there is an equivalence 
$$
\| \hat K_X^\eps(u) - K(u) \|_2 \sim \frac{1}{\sqrt{N \Delta}} + \Delta.
$$
Therefore, in this case the 
choice $\Delta \sim N ^{-1/3}$ implies the optimal sub-sampling strategy. 
Indeed, in dimension 1, given any continuous and piecewise $C^1$ positive definite function $K(u)$, 
such that $ | K(u) | \leq f(u) $ where $f(u)$ is continuous, decreasing, and integrable, there is a strictly stationary centered Gaussian process $X_t \in \bR$ with covariances $K(u)$ (see \cite{shepp72}). 
The 2nd order moments of $X_t$ must then have integrable decorrelation rate proportional to 
$f$, as seen above. For standard examples of stationary 1-dimensional Gaussian processes, the estimators 
$\hat K_X^\eps(u) $ do achieve $L^2$ -errors of estimation which are actually equivalent to 
$(N \Delta)^{-1/2} + \Delta$. 
See for instance \cite{Azen1,Azen2} where the case of the Ornstein-Uhlenbeck process is studied in detail.
\end{proof}

\section{Accuracy of Sub-sampled Moments Estimators under Indirect Observability}
\label{sec:indirect}
We now study $L^2$-consistency for \textit{observable} sub-sampled empirical estimators of lagged covariance matrices. The $L^4$-speed $\rho(\eps)$ at which $Y_t^\eps$ approximates $X_t$ is of course application specific. 
One of our goals was to analyze how $\rho(\eps)$ determines optimal sub-sampling schemes for collecting observable data. This is achieved in the following result. 
\begin{mytho} \label{ThOptimalScheme}
Let $X_t = \lim_{\eps \to 0} Y^\eps_t $ be an indirect observability setup in $ \bR^r$ as in section \ref{IOS}. 
Call $\mu$ and $K(u)$ the mean and lagged covariance matrices of $X_t$, respectively.
Let $\rho(\eps)$ be the $L^4$-speed at which $Y_t^\eps$ approximates $X_t$ for all $t$. 
Let $\nu$ be an upper bound of all the $ \| X_t \|_4 $ and $\| Y^\eps_t \|_4$.
Assume that $X_t$ has integrable decorrelation rate $f(T)$.

For each adaptive sub-sampling scheme $\Delta(\eps), N(\eps)$ verifying $\Delta(\eps) \to 0$ and 
$N(\eps) \Delta(\eps) \to \infty$ as $\eps \to 0$, formula \eqref{DefCovYep} defines observable sub-sampled 
estimators $\hat K_Y^\eps(u) $ of $K(u)$. For each $u$, let $\mathcal{E}(u)$ be the set of all 
such observable estimators.

Then among all the observable estimators $\hat K_Y^\eps(u) $ in $\mathcal{E}(u)$, the \textit{best achievable} 
$L^2$-speed of convergence to the true $K(u)$ as $\eps \to 0$ is equivalent to some constant multiple of $\rho(\eps)$ 
and is achieved by any sub-sampling scheme of the form 
\begin{equation} \label{optimizedScheme}
N(\eps) \sim \frac{1}{\rho(\eps)^3}, \qquad
\Delta(\eps) \sim \frac{1}{N(\eps)^{1/3}} \sim \rho(\eps).
\end{equation}
For any such sub-sampling scheme and any fixed A, one indeed has 
\begin{eqnarray}
\| \hat K_Y^\eps(u) - K(u) \|_2 &\leq& C \rho(\eps) \quad \text{for all~~} 0 \leq u \leq A , \label{optimEq1} \\
\| \bar{Y}^\eps - \mu \|_2 &\leq& C \rho(\eps),
\label{optimEq2}
\end{eqnarray}
for some constant $C$ determined by $A, \nu$ and $I(f) = \int_1^{+ \infty} f(T) dT$.

Let $S(\eps) \sim N(\eps) \Delta(\eps)$ be the \textit{observational time}span gathering the time 
indexes $n \Delta$ of all the observables $Y^\eps_{n \Delta}$ involved in the estimator $\hat K_Y^\eps(u)$. 
Each optimized sub-sampling scheme of the form \eqref{optimizedScheme} also minimizes the rate at which 
$S(\eps) \to \infty$ as $\eps \to 0$.
\end{mytho}
Theorem \ref{ThOptimalScheme} will be proved right after the following more technical proposition.
\begin{mypro}\label{ThObservable}
Let $X_t = \lim_{\eps \to 0} Y^\eps_t $ be an indirect observability setup in $ \bR^r$ verifying all the assumptions of 
Theorem \ref{ThOptimalScheme}. 
Let $\rho(\eps)$ be the $L^4$-speed at which $Y_t^\eps$ approximates $X_t$ for all $t$. 
Call $\mu$ and $K(u)$ the mean and lagged covariance matrices of $X_t$.
Moreover, fix any adaptive sub-sampling scheme $\Delta(\eps), N(\eps)$ such that 
$\lim_{\eps \to 0} N(\eps) = \infty$ and 
$\Delta(\eps) \sim N^{-1/3}(\eps)$.

By formulas \eqref{DefCovYep} and \eqref{barY}, this sub-sampling scheme defines \textit{observable} estimators 
$\hat{K}^\eps_Y(u) $ and $\bar{Y}^\eps$ of $K(u)$ and $\mu$, respectively. 
Then all these estimators are $L^2$-consistent as $\eps \to 0$, 
and for any fixed $A >0$ one has the uniform $L^2$-speeds of convergence 
\begin{eqnarray}
\| \hat K_Y^\eps(u) - K(u) \|_2 &\leq& 4 \nu \rho(\eps) + \frac{c}{N^{2/3}(\eps)}\quad \text{for all} \quad 0 \leq u \leq A, \label{ThObsEq1} \\
\| \bar{Y}^\eps - \mu \|_2 &\leq& \rho(\eps) + \frac{c}{N^{2/3}(\eps)},
\label{ThObsEq2}
\end{eqnarray}
where $c$ is a constant determined by $A, \nu, I(f)$.
\end{mypro}
\begin{proof}
The bound in \eqref{ThObsEq1} is a direct consequence of the bounds obtained for 
$\| \hat K_Y^\eps(u) - \hat K_X^\eps(u)\|_2$ and $\| \hat K_X^\eps(u) - K(u) \|_2$ in propositions \ref{ThCompareYX} and \ref{ThUnobservable}.
Similar arguments prove the bound in \eqref{ThObsEq2}.
\end{proof}
We can now come back to proving theorem \ref{ThOptimalScheme}.
\paragraph{Proof of theorem \ref{ThOptimalScheme}.}
Taking $\Delta \sim N^{-1/3}$, and hence $S \sim N^{2/3}$, we apply proposition \ref{ThObservable}. The bound $B(\eps)= 4 \nu \rho + c N^{-1/3}$ in equation \eqref{ThObsEq1} is larger than $4 \nu \rho$, so that to minimize $B(\eps)$ up to multiplicative constants there is no asymptotic advantage in taking $N^{-1/3} \ll \rho$. To simultaneously minimize (up to multiplicative constants) both $S \sim N^{2/3}$ and $B(\eps)$, a natural choice is then to take $N^{-1/3} \sim \rho(\eps)$ which yields $N \sim \rho^{-3} $ and hence $\Delta \sim \rho$. Then, $B(\eps)$ is inferior to $(4 \nu +c) \rho$ which proves the equation \eqref{optimEq1}. The bound on $ \| \bar{Y}^\eps - \mu \|_2$ provided by \eqref{ThObsEq2} is then equal to $(1+c) \rho$ which proves the equation \eqref{optimEq2}.

To show that the $L^2$-speeds of convergence in \eqref{optimEq1} and \eqref{optimEq2} cannot be generically improved for observable sub-sampled covariance matrix estimators in the class $\mathcal{E}(u)$, consider a 1D centered Gaussian process $X_t$ with preassigned covariance function 
$K_X(u)$
assumed to be piecewise $C^1$ and to decay at an integrable rate $f(u)$ as $u \to \infty$. 
Next, define $Y^\eps_t = X_t + \rho(\eps) X_t$ where $\rho(\eps)$ is any function such that 
$\lim_{\eps \to 0} \rho(\eps) = 0 $. Then $\| Y^\eps_t - X_t \|_4 = C \rho(\eps)$ and all the hypotheses of proposition 
\ref{ThObservable} are satisfied. 
Moreover $Y^\eps_t$ is a centered Gaussian process with lagged covariances $K_X(u) (1 + 2\rho(\eps) + \rho^2(\eps))$.
Then, for any adaptive sub-sampling scheme $\Delta(\eps), N(\eps)$ the norm 
$h(\eps)^2 = \bE [( \hat K_Y^\eps(u) - K(u) )^2]$ can be explicitly computed in terms of $N, \Delta, \rho$ and
moments of $X_t$, and it can be checked that one always has 
$\liminf_{\eps \to 0} {h(\eps)}/{\rho(\eps)} > 0 $. Since we already know that the optimized explicit sub-sampling scheme 
\eqref{optimizedScheme} does yield $h(\eps) \sim \rho(\eps)$, this class of specific Gaussian examples proves generic optimality for the announced speed of convergence.

Previous results presented in this section have the following direct consequence for sparsely parametrized stationary processes. 
\begin{mytho} \label{ThParamEst}
Let $X_t = \lim_{\eps \to 0} Y^\eps_t $ be an indirect observability setup in $ \bR^r$ verifying all the assumptions of Theorem \ref{ThOptimalScheme}. Let $\rho(\eps)$ be the $L^4$-speed at which $Y_t^\eps$ approximates $X_t$ for all $t$. Assume also that $X_t$ is \emph{sparsely parametrized} by $\btheta \in \Theta \subset \bR^p$, as in definition \ref{sparse.def}. \\
Then there exist observable estimators $\hat{\btheta}^\eps$ converging in probability to $\btheta$ as $\eps \to 0$. After selecting any  sub-sampling scheme $N(\eps), \Delta(\eps)$ of the form  \eqref{optimizedScheme}, one can  construct these estimators by an expression of the form 
\begin{equation}
\label{thetahat}
\hat{\btheta}^\eps = G(\hat{\Psi}^{\eps}) \;\; \text{with}\;\; \hat{\Psi}^{\eps} = \left[ \hat{\Psi}_1^\eps, \ldots, \hat{\Psi}_p^\eps \right], 
\end{equation}
 where $G : \bR^p \to \bR^p$ is a fixed smooth function  and each $\hat{\Psi}_j^\eps$ is an observable  sub-sampled empirical estimator of the lagged 2nd order moment $\Psi_j$ of $X_t$.
Moreover, if $\Theta$ is included in some  known euclidean ball $\Lambda $ of finite radius, the truncated observable estimators $1_{\Lambda}(\hat{\btheta}^\eps ) \hat{\btheta}^\eps$ then converge in  $L^2$ norm  to $\btheta$ as $\eps \to 0$, with $L^2$ speed of convergence faster than $C \rho(\eps)$, for some constant $C$.
\end{mytho}
\begin{proof}
Definition \ref{sparse.def}  states that $\btheta = G(\Psi)$ where the vector $\Psi = \Psi_1, \ldots, \Psi_p $ involves  $p$ lagged moments of order  $\leq 2$ of the unobservable process $X_t$ and $G$ is a smooth function. Select  an optimized adaptive sub-sampling scheme of the form  \eqref{optimizedScheme}, and let  $\hat{\Psi}_j^\eps$ be the  associated observable  sub-sampled empirical estimator of the lagged 2nd order moment $\Psi_j$.  By theorem \ref{ThOptimalScheme}, as $\eps \to 0$, the estimators $\hat{\Psi}_j^\eps$ converge to $\Psi_j$ in $L^2$, and hence also converge to  $\Psi_j$ in probability. Since convergence in probability is preserved by smooth functions, the estimators  $\hat{\btheta}^\eps$ by equation \eqref{thetahat} must then converge in probability to $\btheta$ as $\eps$ to zero.
For uniformly bounded random vectors, convergence in probability always implies convergence in $L^2$, and this proves the last statement of the theorem.   
\end{proof}

\section{Applications to indirectly observable multi-dimensional diffusions }
\label{sec:appl}
We now present several examples of stationary multi-dimensional diffusions $X_t$ naturally embedded in indirect observability frameworks. A key assumption to prove $L^2$-consistency for natural observable estimators of the lagged covariances $K(u)$ of $X_t$ is to require the lagged moments of order $\leq 4$ of $X_t$ to decay at  integrable decorrelation rate for large lags,  as specified in equation \eqref{decay4}.

Published literature does not provide easy generic conditions on SDEs coefficients guaranteeing that the associated diffusion $X_t$ has integrable decorrelation rate as specified in \eqref{decay4}. Quite relevant exponential decay bounds for the transition density 
$p_{\btheta}(t,x,y) $ as $t \to \infty$ have been given in \cite{aron, pe,ns,davies}, but more precise bounds on $p_{\btheta}$ are needed to generically validate the integrable decorrelation rates on lagged moments of order 4 as required by equation \eqref{decay4}.
Here we do not attempt to solve these technical questions for general classes of diffusions. Instead, we will simply list a few interesting examples of diffusions for which our assumptions can either be directly verified, or are quite plausibly conjectured to be true, as can be also tested by numerical simulations.

\paragraph{Gradient Diffusions.} In section 7.1 of \cite{Hairer10}, M. Hairer discusses the \lq\lq{}gradient diffusions \rq\rq{} $X_t$ in $\bR^r$ driven by the SDE 
$$
dX_t = - \nabla Q(X_t) dt + \sigma dW_t,
$$
where $Q(x)$ is a smooth \lq\lq{}potential\rq\rq{} defined for $x \in R^r$, and $\sigma$ is a constant $r \times r$ matrix. 
The potential $Q$ is also assumed to behave as a polynomial at infinity, i.e. there are constants $c, \, C, \, k$ such that 
\[
c|x|^{2k} \le Q(x) \le C|x|^{2k}, \quad
\langle x, \nabla Q(x) \rangle \ge c|x|^{2k}, \quad
\left| D^2 Q(x) \right| \le C|x|^{2k-2},
\]
where $\nabla$ is the gradient operator and $D$ denotes 1st order differentiation operators.
Under these conditions, \cite{Hairer10} proves that the probability distribution of $X_t$ given $X_0 = x$ converges to the stationary probability distribution of $X_t$ at exponentially fast speed as $t \to \infty$, and that $X_t$ verifies the classical Doeblin property. Results of \cite{Hairer10} imply the exponentially fast decorrelation 
$$
| E(G H) - E(G) E(H) | \leq e^{-\gamma T} ,
$$
where $\gamma >0$ is a constant, but only for random variables $G, H$ of the form $G= f(X_s) \phi(X_t) $ and 
$H= F(X_u) \Psi(X_v)$ with $f,\phi, F, \Psi $ bounded and $s \leq t < t+T \leq u \leq v$. Whenever the Aronson bounds \cite{aron} on transition densities hold, this decorrelation inequality can be extended to $G = X_s X_t $ and $H = X_u X_v $ so that the \lq\lq{}gradient diffusions\rq\rq{} $X_t$ provide a class of stationary diffusions with integrable decorrelation rate $f(T) = e^{- \gamma T}$ and finite moments of order 4.
Whenever a given multidimensional diffusion process $X_t$ verifies all our assumptions on the unobservable process, the simplest example of observable processes $Y^{\eps }_t$ associated to $X_t$ is generated by local smoothing , i.e. 
$$
Y_t^\eps = \frac{1}{\eps} \int\limits_{t-\eps}^t X_s ds.
$$
Such observable processes are often generated by sensors recording short-term averages of high frequency input data. 
See \cite{Azen1} for a detailed study of this case when $X_t$ is a Gaussian diffusion.

\paragraph{Volatility Processes and Heston joint SDEs.} A striking example of indirect observability is quite ubiquitous in stochastic modeling of joint price and volatility for stockmarket data. The well known Heston model (see \cite{heston}) links the price $S_t$ and the squared volatility $V_t$ of an asset by parametrized joint SDEs of the form 
\begin{eqnarray}
dS_t &=& \mu S_t dt + \sqrt{V_t} S_t dW_1, 
\nonumber \\
dV_t &=& \kappa(\theta - V_t) dt + \sigma \sqrt{V_t} dW_2,
\nonumber
\end{eqnarray}
where the unknown positive parameters $\mu, \kappa,\theta,\sigma $ need to be estimated from asset price data 
$S_t$ only, since the squared instantaneous volatility $V_t$ is not directly available, and plays the part of our 
unobservable process $X_t \equiv V_t$ . 
In our indirect observability framework, volatility approximations $ Y_t^\eps $ based either on prices $S_t$ or on 
observed option prices become the observable processes. A classical volatility approximation is the \lq\lq{}realized volatility\rq\rq{} given by the sum of squared returns 
\begin{equation}\label{}
Y_t^\eps = \frac{1}{M\eps} \sum\limits_{k=1}^M (R_{t_k} - R_{t_{k-1}})^2, \quad \text{where~~} dR_t = \frac{dS_t}{S_t},
\end{equation}
In this equation, the time step size $ t_k - t_{k-1} = \eps$ and the window size $M= M(\eps)$ are user selected. 
The $L^2$ convergence of realized volatility to instantaneous volatility as $\eps \to 0 $ is studied in \cite{Nielsen1,Nielsen2}. In a companion paper to be published in \cite{ourfuturepaper}, we have proved that the pair $( Y_t^\eps , V_t,)$ verifies all the hypotheses of our indirect observability setup, and we have completed a detailed analysis of parameter estimation under indirect observability for generic Heston models (see also \cite{prenthesis}).

\paragraph{Averaged Multiscale Stochastic Systems.} 
Consider a \lq\lq{}slow-fast\rq\rq{} joint SDEs system (see \cite{book:past08} 
for overview and references) involving a (small) scale parameter $\eps$ and given by 
\begin{eqnarray}
dx_t &=& a(x_t,y_t) dt + b(x_t,y_t) dW_1(t) , 
\label{xt} \\
dy_t &=& \frac{1}{\eps} c(x_t,y_t) dt + \frac{1}{\sqrt{\eps}} d(x_t,y_t) dW_2(t),
\label{yt}
\end{eqnarray}
where $W_1(t)$ and $W_2(t)$ are independent Brownian motions and the coefficients $a,b,c,d$ are bounded 
smooth functions of $x,y$. 
Note that the diffusions $x_t$, $y_t$ actually depend on the scale parameter $\eps$.
Assume that for any fixed $x$, the \lq\lq{}fast\rq\rq{} SDE driving $y_t$ has a stationary distribution $q(y|x)$ verifying 
$\bE_q \, a(x,y) \ne 0$. Then under mild complementary conditions on $a,b,c,d$, and as $\eps \to 0$, the process $x_t$ converges in probability to the \lq\lq{}reduced dynamics\rq\rq{}
\begin{equation}
\label{Xt}
dX_t = A(X_t) dt + B(X_t)B^*(X_t) dW_1(t),
\end{equation}
where $A = \bE_q \, a(x,y)$ and $BB^* = \bE_q \, b b^*$. Convergence in probability implies $L^4$ convergence 
for variables bounded in $L^4$, and the $L^4$ convergence of $x_t$ to $X_t$ is proved in \cite{book:past08} for 
periodic coefficients and $x_t, y_t $ on a torus. In practical applications, one essentially wants to parametrize the slow asymptotic SDE \eqref{Xt} driving the unobservable process $X_t$, and the only realistically accessible data are 
generated by the approximating process $x_t$, since for small $\eps$ the $y_t$ data are too noisy to be reliably acquired. Hence the slow process $x_t$ plays the role of the observable process $Y_t^\eps$.

There are many practical applications when $A(X_t)$ has polynomial nonlinearities. When $X_t$ is one-dimensional this trivially corresponds to the case of gradient diffusions discussed earlier in this section. For multi-dimensional $X_t$ one can conjecture that the mixing rates for the process $X_t$ must obey exponential decay unless equation \eqref{Xt} posesses some unusual properties (e.g. special symmetries or existence of conserved quantities). It is then quite reasonable to expect exponential convergence to the equilibrium distribution and, thus, exponentially fast decorrelation rates of  lagged 4th moments for large lags. Moreover, exponentially fast decorrelation rates have been demonstrated numerically in many practical examples. 
So we expect that our key assumptions on the unobservable and observable processes will be satisfied for many multiscale examples of the form \eqref{Xt} and \eqref{xt}, respectively.

\section{Conclusions} 

For stationary processes $X_t \in \bR^r$ which are not directly observable, but can be approximated in $L^4$ by observable processes $Y^{\eps}_t$ as $\eps \to 0$, we have developed a mathematical framework where the unknown vector $\btheta \in \bR^p$ parametrizing $X_t$ can be consistently and efficiently estimated from adequately sub-sampled observations of $Y^{\eps}_t $. The present paper extends substantially several of our earlier results  \cite{Azen1,Azen2,Azen3} to non-Gaussian stationary processes $X_t$.

We have focused on the \lq{}\lq{}sparsely parametrized\rq{}\rq{} situations where $\btheta \in \bR^p$ is a (generally non explicit) smooth function $G(\Psi_1,\ldots,\Psi_p)$ of $p$ lagged moments of order $\leq 2$ of $X_t$. We conjecture that this holds true when $X_t$ is a stationary multi-dimensional diffusion provided the matrix diffusion coefficients $a(x,\btheta)$ and the drift $b(x,\btheta)$ of $X_t$ are analytic in $x$ and $\btheta$.

The above setup leads us to study the class of parameter estimators of the form $\hat{\btheta}^\eps = G(\hat{\Psi}^{\eps})$,
 where $\hat{\Psi}^{\eps}$ is an observable empirical estimator of $\Psi$  based on the $N(\eps)$ sub-sampled
 observable data  $Y^\eps_{{n \Delta(\eps)}}$, with $n= 1 \ldots N(\eps)$. For parameter estimators such as
 $\hat{\btheta}^\eps$, analysis of consistency and speed of convergence is essentially  equivalent  to a similar but more
 technical analysis for observable subsampled estimators $\hat{K}_{Y^{\eps}(u)}$ of  the lagged covariances $K(u)$
 of $X_t$. Note that for $u>0$, estimators $\hat{K}_{Y^{\eps}(u)}$ involve only \emph{non-vanishing time lags}
 $u(\eps)$ (with $u(\eps) \to u>0$ as $\eps\to 0$), since vanishing time lags decrease robustness to data perturbations
(see \cite{Azen1,Azen3}).

We explicitly determine how to best choose the sub-sampling time step $\Delta(\eps)$ and the number $N(\eps)$ of observations in terms of the $L^4$ distance 
$\rho(\eps) = || Y^{\eps}_t - X_t ||_4$ (see equation \eqref{optimizedScheme}). Our asymptotically optimal sub-sampling schemes $N(\eps) \sim \rho^{-3}(\eps)$ and $\Delta(\eps) \sim \rho(\eps)$ are constructed  to simultaneously  minimize the amplitude  of estimation errors as well as the computational/observational  complexity due to both the  number $N(\eps)$ and the time span $S(\eps)=N(\eps) \Delta(\eps)$ of  sub-sampled  observable data. Indeed, in many practical situations such as joint dynamic modeling of observable stock prices and unobservable volatilities, both $N(\eps)$ and $S(\eps)$ must remain  rather moderate, even for intraday data. 
Our optimal sub-sampling results rely on a key hypothesis, stating that for $s \leq t \leq u \leq v$ the random variables    $X_s X_t$ and $X_{u+T} X_{v+T}$ decorrelate at an integrable rate $f(T) \to 0 $ when $T \to \infty$ (see \eqref{decay4}). This is generally true for the many practical situations where $X_t$ is a stationary multi-dimensional diffusion with exponentially fast mixing.

When $X_t$ is sparsely parametrized and has integrable decorrelation rate $f$, we 
prove that as $\eps \to 0$, the sub-sampled observable estimators of lagged covariances determined by our optimal sub-sampling scheme \eqref{optimizedScheme} converge in $L^2$ to the true lagged covariances of $X_t$, with $L^2$-speeds of convergence faster than  $c \rho(\eps)$ for some constant $c$, and that the decorrelation rate $f$ only affects the constant $c$ via the integral $I(f)$ of $f$. Our associated observable subsampled parameter estimators 
$\hat{\btheta}^\eps = G(\hat{\Psi}^{\eps})$ are then consistent in probability when $\eps \to 0$. In practical applications, the unknown $\btheta$ is of course a priori bounded, and then a natural   truncation of the estimators $\hat{\btheta}^\eps$ guarantees their $L^2$-convergence to $\btheta$ at $L^2$-speed faster than some constant multiple of $\rho(\eps)$.

Our work thus points out the pragmatic impact of numerical methods enabling fast evaluation of $\rho(\eps)$, to help determine nearly optimal sub-sampling schemes, as well as for computing approximate error bars on parameter estimators. We will study numerical applications of our approach for non-Gaussian $X_t$ in subsequent papers.

Our indirect observability study has strong practical consequences for a broad range of applications. Let us mention two examples.
In financial mathematics, our indirect observability setup potentially applies to many  stochastic volatility models  driving the price and volatility of a given asset. The observable $Y^{\eps}_t$ can then be a realized volatility estimated on a time window depending on $\eps$, and $X_t$ is the unobservable instantaneous volatility. For the well known Heston joint SDEs, our approach has enabled the construction of consistent and efficient explicit parameter estimators based on optimally sub-sampled realized volatility data \cite{our future paper, prenthesis}.

A second class of examples concerns complex multiscale systems driving atmospheric or ocean dynamics. In this case the numerical datasets generated by known high dimensional fluid evolution models can be analyzed by artificially inserting a small scaling parameter $\eps$ into the model to further accelerate the fast variables and numerically analyze 
(as $\eps$ varies) the behavior of parameter estimators for key parameters of the slow dynamics. 
This computer intensive version of our approach should yield both a concrete and efficient optimal sub-sampling scheme as well as approximate error bars for our parameter estimators. We will present detailed actual examples in further publications.

\subsubsection*{Acknowledgements.} 
I.T. and R.A. were supported in part by the NSF Grant DMS-1109582.

\appendix
\section{$L^2$- consistency for unobservable estimators of lagged 2nd order moments}
\label{ProofUnobservable}
In this appendix we present a detailed proof of theorem \ref{ThUnobservable}, which addresses the $L^2$- consistency results for the unobservable sub-sampled empirical estimators $\bar{X}^\eps$ and $\hat K_X^\eps(u)$ of means and lagged covariances. The hypotheses and notations are those of Theorem \ref{ThUnobservable}. 
Replacing $X_t$ by the centered process $X_t - \mu$ and setting $\mu=0$ is a trivial change in the proof , so we only need to consider 
the case where all $X_t$ are \textit{centered} and $\mu = 0$. 

\textbf{Step 1. Sums of decorrelation values.} 
For all $D>0$ and $j \geq 1$ one has $D f(jD) \leq \int_{(j-1)D}^{j D} f(T) dT $ since the decorrelation rate $f(T)$ is decreasing. This implies 
\begin{equation}\label{sum.fjD}
\sum_{j=1}^{\infty} f(j D) \leq \frac{1}{D}\sum_{j=1}^{\infty} \int_{(j-1)D}^{j D} f(T) dT = I(f) / D.
\end{equation}
Define the function $g(q, D)$ for all integers $q \geq 2$ and all $D>0$ by 
\begin{equation}\label{g(q)}
g(q, D) = \sum_{1 < m < n \leq 1+q} f( (n - m) D) = \sum_{j=1}^{q-1} j f(j D) .
\end{equation}
Due to \eqref{sum.fjD}, the following inequality holds for all $D>0$ and $q \geq 2 $ 
\begin{equation}\label{bound.g(q)}
g(q, D) \leq (q-1) \sum_{j=1}^{q-1} f(j D) \leq ( q-1) I(f) /D.
\end{equation}

\textbf{Step 2. Sub-sampled empirical means converge in $L^2$.} 
Fix an integer $j \in [ 1 \ldots r]$. Denote the $j$-th coordinates of $X_{n \Delta}$ and of the empirical mean estimator $\bar{X}^\eps$ by
\[
U_n = X_{n \Delta}(j) \quad \text{and} \;\; \bar{X}^\eps (j) = \frac{1}{N} (U_1 + \ldots + U_N).
\]
With the notation $s_j ^2 = \bE(U_n^2)$, this implies
\begin{equation}\label{EbarX2} 
N^2 \bE \left[(\bar{X}^\eps (j))^2 \right] = N s_j ^2 + 2 \sum_{1 \leq m < n \leq N } \bE[U_m U_n] .
\end{equation} 
Applying the decorrelation hypothesis \eqref{decay4} and the relations \eqref{g(q)}, \eqref{bound.g(q)}, we obtain
\[
\left| \sum_{1 \leq m < n \leq N } E(U_m U_n) \right| \leq \sum_{1 \leq m < n \leq N } f((n - m) \Delta) = 
g(N-1, \Delta) \leq I(f) \frac{N}{\Delta}.
\]
The definition of the $L^q$-norm also implies $s_j \leq \| X_t \|_2 \leq \| X_t \|_4 =\nu$.
Hence, \eqref{EbarX2} implies 
\begin{equation*}
\left( \| \bar{X}^\eps (j) \|_2 \right)^2 \leq \frac{\nu ^2 }{N} + \frac{ 2 I(f) }{N \Delta}.
\end{equation*}
For any $( r_1 \times r_2)$ random matrix $M$, and any $q \geq 1$ our definition of the norm $\|M\|^q$ implies 
\begin{equation} \label{boundMq}
\| M \|_q \leq (r_1 r_2)^{1/q} \max_{i, j} \| M_{i, j} \|_q.
\end{equation}

The inequality $\| \bar{X}^\eps \|_2 \leq \sqrt{r} \max_{j} \| \bar{X}^\eps (j) \|_2$, 
then yields, due to \eqref{boundMq}, 

\begin{equation*}
\| \bar{X}^\eps \|_2 \leq 
\sqrt{r} \left( \frac{\nu^2}{N} + \frac{ 2 I(f) }{N \Delta} \right)^{1/2} \leq 
\frac{1}{\sqrt{N \Delta} } \left( \nu (r \Delta)^{1/2} + (2 r I(f) )^{1/2} \right).
\end{equation*}
Since $\Delta(\eps) \to 0 $ this proves the $L^2$-bound in \eqref{L2L4speed} when $X_t$ is centered and hence in general.

\textbf{Step 3. Sub-sampled empirical means converge in $L^4$.}
Basic algebra yields the identities 
\begin{equation}\label{EbarX4}
N^4 \left( \bar{X}^\eps (j) \right) ^4 = \sum_{m, n, a, b \in [ 1 \ldots N ]} U_a U_b U_m U_n = S_1 + S_2 + 2 S_3 +24 S_4,
\end{equation}
where the sums $S_1$, $S_2$, $S_3$, and $S_4$ are defined by
\begin{eqnarray}
S_1 &=& \sum_{1 \leq m \leq N } U_m^4 ,
\nonumber \\
S_2 &=& \sum_{1 \leq m < n \leq N } \left[ 2 U_m^2 U_n^2 + U_m^3 U_n + U_m U_n^3 \right] ,
\nonumber \\
S_3 &=& \sum_{1 \leq a < m < n \leq N } 
\left[ U_a^2 U_m U_n + U_a U_m^2 U_n + U_a U_m U_n^2 \right] ,
\nonumber \\
S_4 &=& \sum_{1 \leq a < b < m < n \leq N } U_a U_b U_m U_n .
\nonumber 
\end{eqnarray}
Due to the assumption that the $L^4$ norm of $X_t$ is bounded uniformly by $\nu$,
one clearly has $|\bE( S_1) | \leq N \nu^4$ and $ | \bE(S_2) | \leq 4 N^2 \nu^4$.

Since we are considering the centered process $X_t$, $E(U_n) \equiv 0$, and 
for $a < m < n$ the decorrelation hypothesis implies
\[
\left| \bE(U_a^2 U_m U_n) \right| = \left| \bE(U_a^2 U_m U_n) - \bE(U_a^2 U_m) E (U_n) \right| \leq
f( (n-m) \Delta ).
\]
Similarly, one shows that 
$$
\left| \bE(U_a U_m^2 U_n) \right| \leq f( (n-m) \Delta ) \text{~~and~~} 
\left| \bE(U_a U_m U_n^2) \right| \leq f( (m-a) \Delta ).
$$ 
These bounds and definition \eqref{g(q)} yield (for $N \geq 3$)
\[
\left| \bE(S_3) \right| \leq \sum_{1 \leq a < m < n \leq N ]} \; 
\left[ 2 f( (n-m) \Delta ) + f( (m-a) \Delta ) \right] = 3 \sum_{2 \leq m \leq N-1 ]} g(m, \Delta)
\]
which implies, due to the bound \eqref{bound.g(q)}, 
\[
\left| \bE(S_3) \right| \leq 3 I(f) \sum_{2 \leq m \leq N-1 ]}(m-1) \leq \frac{3 I(f)}{2} \frac{N^2 }{\Delta}.
\]
As above, one also has $ | \bE(U_a U_b U_m U_n) | \leq f( (b-a) \Delta ) $ for $a < b < m < n $. 
The expressions of $S_4$ and $g(m, \Delta)$ then yield (for $N \geq 4$)
\[
\left| \bE(S_4) \right| \leq \sum_{1 \leq a <b <m <n \leq N ]} f( (b-a) \Delta ) =
\sum_{3 \leq m <n \leq N} g(m, \Delta) = \sum_{3 \leq m \leq N-1} (N-m) g(m, \Delta).
\]
Therefore, due to \eqref{bound.g(q)} we obtain for $N \geq 4$
\[
\left| \bE(S_4) \right| \leq \sum_{3 \leq m \leq N-1} (N-m) I(f) \frac{m}{\Delta} \leq 
\frac{I(f)}{4} \frac{N^3}{\Delta}.
\]
Finally, the bounds on $| \bE(S_k) |$, and equation \eqref{EbarX4} entail
\begin{equation}\label{boundEbarX4}
\bE \left[ \left( \bar{X}^\eps (j) \right)^4 \right] \leq \frac{5 \nu^4 }{N^2} +
\frac{3 I(f)}{N^2 \Delta} + \frac{ 6 I(f)}{N \Delta} \leq \frac{C}{N \Delta} 
\end{equation}
for some explicit constant $C$, since $N(\eps) \to \infty$ and $\Delta(\eps) \to 0 $ with $N(\eps) \Delta(\eps) \to \infty$. In particular for $\eps$ small enough, one can clearly take $C = 7 I(f)$.
Therefore, equations \eqref{boundMq} and \eqref{boundEbarX4} imply 
\[
\| \bar{X}^\eps \|_4 \leq r^{1/4} \max_{i, j} \| \bar{X}^\eps \|_4 \leq \frac{(r C)^{1/4}}{(N \Delta)^{1/4}} 
\]
which proves the $L^4$-bound in \eqref{L2L4speed}.

\textbf{Step 4. Convergence of empirical lagged covariance matrices estimators.} 
Introduce the short-hand notations $V_n = X_{n \Delta}$ and 
\begin{eqnarray}
\bar{V}_N &=& \bar{X}^\eps = \frac{1}{N} \sum_{n=1}^N \, V_n ,
\label{barV} \\
\tau \bar{V}_N &=& \tau \bar{X}^\eps = \frac{1}{N} \sum_{n=1}^N \, V_{n+ \kappa},
\label{taubarV} \\
W_N &=& \frac{1}{N} \sum_{n=1}^N \, V_n V_{n+ \kappa}^* .
\label{WN}
\end{eqnarray}
From the definition \eqref{DefCovXep}, the covariance matrix estimators $\hat{K}^\eps_X(u) $ can be rewritten as
\begin{equation}\label{hatK1}
\hat{K}^\eps_X(u) = W_N - \bar{V}_N (\tau \bar{V}_N)^* .
\end{equation}
First, we evaluate the term $\bar{V}_N (\tau \bar{V}_N)^* $ in the equation above. 
Impose $0 \leq u \leq A$ for some fixed $A$. Thus, by construction
\[
\| \bar{V}_N - \tau \bar{V}_N \|_4 \leq 2 \kappa \nu / N \leq 2 (u+ \Delta) \nu/N \leq 
2 (1+A) \nu \frac{1} {N \Delta}
\]
and applying the inequality \eqref{lemmaVW} one arrives at the following relation
\begin{equation}\label{barV2}
\| \bar{V}_N (\tau \bar{V}_N)^* - \bar{V}_N (\bar{V}_N)^* \|_2 \leq 
4 (1+A) \nu^2 \frac{1} {N \Delta}.
\end{equation}
Since $\mu = 0$, we also have 
$$ 
\| \bar{V}_N \|_4 \leq \frac{C}{(N \Delta)^{1/4}} ,
$$ 
as proven in Step 3. 
This implies, by inequality \eqref{lemmaVW}, 
\[
\| \bar{V}_N (\bar{V}_N)^* \|_2 \leq 2 \left[ \frac{C}{(N \Delta)^{1/4}} \right]^2 =
\frac{2 C^2}{(N \Delta)^{1/2}}
\]
which yields, due to equation \eqref{barV2}, 
\begin{equation}\label{barVtaubarV}
\| \bar{V}_N (\tau \bar{V}_N)^* \|_2 \leq 
\frac{2 C^2}{\sqrt{N \Delta}} + \frac{4 (1+A) \nu^2 } {N \Delta}.
\end{equation}
By the construction of $\kappa(u,\eps)$, the 
\lq\lq{}discrete\rq\rq{} lag $\kappa \Delta$ is close to continuous lag $u$ and 
$| \kappa \Delta - u | \leq \Delta$.
Since the true lagged covariance matrices $K(u)$ are locally Lipschitz, there is a constant $\lambda = \lambda(A)$ such that for 
all $0 \leq u \leq A$ and all $\eps >0$ 
the following deterministic inequality holds
\begin{equation} \label{du}
\| K(u) - K(\kappa \Delta ) \| \leq \lambda | u - \kappa \Delta | \leq \lambda \Delta.
\end{equation}

Next, we compare the term $W_N$ in the expression of the covariance estimator \eqref{hatK1}, with the true covariance matrix $K(\kappa \Delta)$ evaluated at the \lq\lq{}discretized \rq\rq{} time lag $\kappa \Delta$. 
Since $X_t$ is stationary, we have $K(\kappa \Delta) = \bE( V_n V_{n+ \kappa}^*) $ for all $n$, and formula \eqref{WN} implies that
\[
W_N - K(\kappa \Delta) = \frac{1}{N} \sum_{n=1}^N \; 
\left( \, V_n V_{n+ \kappa}^* - \bE[ V_n V_{n+ \kappa}^*] \, \right) .
\]

For any two coordinates $i, j \in [1 \ldots r]$ 
denote $T_n= V_n(i)$ and $U_n= V_n(j)$ as the $i$-th and $j$-th coordinates of $V_n$, respectively.
In addition, we also define
\[
H_n = T_n U_{n+ \kappa} - \bE[ T_n U_{n+ \kappa}].
\]
Clearly $\bE[H_n] = 0$ and 
the $(i, j)$ coefficient of the matrix $ M = W_N - K(\kappa \Delta)$ is then
\[
M_{i, j} = \frac{1}{N} \sum_{n=1}^N \, H_n
\]
and
\begin{equation}\label{EMij2}
N^2 \bE[M_{i, j}^2] = \sum_{(m, n) \in [1 \ldots N] } \bE[H_m H_n]. 
\end{equation}
%

%
%
Next, we partition the summation interval in the expression above into two complementary sets, $Q^+$ and $Q^-$, as follows
\begin{itemize}
\item 
$(m,n) \in Q^+$ whenever $ | n - m | > \kappa $ and $m, n \in [1 \ldots N] $,
\item 
$(m,n) \in Q^-$ whenever $ | n - m | \leq \kappa $ and $m, n \in [1 \ldots N]$. 	
\end{itemize}
Due to bounded fourth moments of the process $X_t$
we have $\left| \bE[H_m H_n] \right| \leq 2 \nu^4 $. Moreover, 
\[
\text{cardinal} \, (Q^-) = N + \kappa (2N- \kappa -1) \leq 3 N \kappa, 
\]
and, therefore,
\begin{equation}\label{Q-}
\left| \sum_{(m,n) \in Q^-} \bE(H_m H_n) \right| \leq 6 \nu^4 N \kappa.
\end{equation}
For $(m,n) \in Q^+$, the decorrelation rate of the 2nd order moments yields 
\[
| \bE[H_m H_n] | \leq f( | n-m | \Delta ) , 
\]
so that 
\begin{equation}\label{Q+}
\left| \sum_{(m,n) \in Q^+} \bE[H_m H_n] \right| \leq \sum_{(m,n) \in Q^+} f( | n-m | \Delta ) .
\end{equation}
Thus, relation \eqref{EMij2} and inequalities \eqref{Q-}, \eqref{Q+} imply 
\begin{equation}\label{lastEMij2}
N^2 \bE[M_{i, j}^2] \leq \sum_{(m,n) \in Q^+} f( | n-m | \Delta ) + 6 \nu^4 N \kappa.
\end{equation}
Easy algebra transforms the double sum above into 
\[
\sum_{(m,n) \in Q^+} f( | n-m | \Delta ) = 
2 \sum_{\kappa +1 \leq s \leq N-1} ( N - s ) f( s \Delta ) \leq
2 N \sum_{1 \leq s \leq N-1}f( s \Delta ) \leq 2 I(f) \frac{N}{\Delta}, 
\]
where equations \eqref{sum.fjD} were used in the last inequality. 
Recall that for $0 \leq u \leq A$ and due to the construction of $\kappa$, one also has 
$$
\kappa \leq \frac{u}{\Delta} + \Delta \leq \frac{A+1}{\Delta} .
$$
Substituting the last two expressions into \eqref{lastEMij2} we obtain
\[
\bE[M_{i, j}^2] \leq \left( 2 I(f) + 6 (A+1) \nu^4 \right) \frac{1}{N \Delta }.
\]
By equation \eqref{boundMq} we further obtain
\begin{equation}\label{boundM}
\| W_N - K(\kappa \Delta) \|_2 = \| M \|_2 \leq 
r \left( 2 I(f) + 6 (A+1) \nu^4 \right)^{1/2} \frac{1}{\sqrt{N \Delta} }.
\end{equation}
Using the expression for $\hat{K}^\eps_X(u)$ in \eqref{hatK1} and
the triangle inequality we can write
\begin{equation*}
\| \hat{K}^\eps_X(u) - K(u) \|_2 \leq 
\| W_N - K(\kappa \Delta) \|_2 + \| K(\kappa \Delta) -K(u) \|_2 +
\| \bar{V}_N (\tau \bar{V}_N)^* \|_2 .
\end{equation*}
Combining the three bounds in \eqref{barVtaubarV}, \eqref{du}, and \eqref{boundM}, we obtain, for all $\eps > 0$ and $0 \leq u \leq A$, 
\begin{equation}
\| \hat{K}^\eps_X(u) - K(u) \|_2 \leq \frac{\Gamma }{\sqrt{N \Delta} } + \lambda \Delta,
\end{equation}
where 
\[
\Gamma = 2 C^2 + r \left( 2 I(f) + 6 (A+1) \nu^4 \right)^{1/2} + \frac{ 4 (1+A) \nu^2}{\sqrt{N \Delta}}.
\]
Moreover, for $\eps$ small enough, we can take $C^2 = \sqrt{7r I(f)}$ as discussed in Step 2, and 
${4 (1+A) \nu^2}{(N \Delta)^{-1/2}}$ will become much smaller than $\sqrt{r I(f)}$. 
Therefore, for $\eps$ small enough, one has (using $\sqrt{a+b} \leq \sqrt{a} + \sqrt{b}$)
a simplified expression for the constant $\Gamma$
\[ 
\Gamma \leq \gamma = 8 \sqrt{r I(f)} + 2.5 \nu^2 \sqrt{A+1}.
\]
This concludes the proof of Theorem \ref{ThUnobservable}.

\end{document}